\documentclass[11pt]{article}
\usepackage{amsmath, amssymb, amscd, amsthm, amsfonts}
\usepackage{mathrsfs}
\usepackage{graphicx}
\usepackage[colorlinks, citecolor=blue]{hyperref}
\newtheorem{theorem}{Theorem}[section]
\newtheorem{lemma}[theorem]{Lemma}
\newtheorem{definition}[theorem]{Definition}
\newtheorem{corollary}[theorem]{Corollary}
\newtheorem{proposition}[theorem]{Proposition}
\newtheorem{example}[theorem]{Example}

\newcommand{\al}{\alpha}

\newcommand{\upcite}[1]{\textsuperscript{\cite{#1}}}

\begin{document}
\title{\bf  On Hom-Groups and Hom-Group actions}
\author{\normalsize \bf Liangyun Chen$^1$,  Tianqi Feng$^1$,  Yao Ma$^1$, Ripan Saha$^2$, Hongyi Zhang$^3$}

\date{\small{ $^1$School of Mathematics and Statistics, Northeast Normal University, Changchun,  130024,  CHINA
  \\
  $^2$ Department of Mathematics,
Raiganj University, Raiganj, 733134, West Bengal, INDIA
\\
  $^3$School of Mathematical Sciences, Nankai University,  Tianjin, 300071,  CHINA}}
\maketitle
\date{}

\begin{abstract}
A Hom-group is the non-associative generalization of a group, whose associativity and unitality are twisted by a compatible bijective map. In this paper, we give some new examples of Hom-groups, and show the first and the second isomorphism fundamental theorems of homomorphisms on Hom-groups. We also introduce the notion of Hom-group action, and as an application, we show the first Sylow theorem for Hom-groups along the line of group actions. \end{abstract}
\bigskip

\renewcommand{\thefootnote}{\fnsymbol{footnote}}
\footnote[0]{Corresponding author(R. Saha): ripanjumaths@gmail.com}
\footnote[0]{Supported by  NNSF of China (Nos. 11771069, 11771410 and 11801066).}

\section{Introduction}\label{section-introduction}
Hartwig, Larsson, and Silvestrov \cite{HLS} firstly introduced the notion of Hom-Lie algebras during the study of deformations of the Witt and the Virasoro algebras. In a Hom-Lie algebra, the Jacobi identity is twisted by a linear map, called the Hom-Jacobi identity in {\rm \cite{AS, CKL, CZ}}. A quantum deformation or a $q$-deformation of an algebra of vector fields is obtained when a usual derivation was replaced by a derivation $d_{\sigma}$ that satisfies a twisted Leibniz rule $d_{\sigma}\left(fg\right)=d_{\sigma}f\left(g\right)+\sigma\left(f\right)d_{\sigma}\left(g\right)$, where $\sigma$ is an algebra endomorphism of a commutative associative algebra. The set of $\sigma$-derivations with the classical bracket is a new type of algebra that does not satisfy the Jacobi identity. However, it satisfies Hom-Jacobi identity and it is called Hom-Lie algebras. Because of close relation of discrete and deformed vector fields and differential calculus in {\rm \cite{HLS, LS1, LS2}}, more people pay special attention to the algebraic structure. The corresponding associative algebras called Hom-associative algebras, were introduced in {\rm \cite{MS1}}. A Hom-associative algebra with a bracket $\left[a,b\right]=ab-ba$ is a Hom-Lie algebra. Other-nonassociative objects such as Hom-coalgebras, Hom-bialgebras and Hom-Hopf algebras were introduced and studied in  {\rm \cite{GMMP, MS2, MS3, Ya2, Ya3, Ya4}}.

Much work has been done between Hom-Lie algebras and Hom-Hopf algebras. Due to the lack of Hom-type of notions for groups and group algebras, Hom-groups which naturally are appearing in the structure of the group-like elements of Hom-Hopf algebras came into the context of Hom-type objects to complete some relations {\rm \cite{LMT}}. Having this notion will help us to find out more about Hom-Hopf algebras. The twisting map $\al$ of a Hom-group $G$ does not need to be invertible in the original works {\rm \cite{MH, LMT}}, however many interesting results in {\rm\cite{Mohammad Hassanzadeh}}, including the main results of this paper, are obtained in the case $\al$ is invertible. The author in {\rm  \cite{Mohammad Hassanzadeh, MH}} has introduced some basics of Hom-groups, their representations, Hom-group (co)homology, and Lagrange's theorem for Hom-groups. In \cite{JMS}, the authors defined a notion of Hom-Lie groups and introduced an action of Hom-group on smooth manifolds.

In this paper, we assume that the twisting map $\al$ is bijective. Then we promote the results of groups to Hom-groups. In Section \ref{sec2}, we introduce some new examples of Hom-groups. In Section \ref{sec2.5}, we discuss some fundamental notions of Hom-groups such as Hom-subgroups and the law of cancellation of Hom-groups. In addition, we show the intersection of two Hom-subgroups is a Hom-group. In Section \ref{sec3}, we define Hom-normal subgroups by Hom-cosets. Then we give Hom-quotient groups and prove that a Hom-quotient group is a Hom-group. In Section \ref{sec4}, we use Hom-quotient groups  to show that the fundamental theorem of homomorphisms of Hom-groups is valid. Then we obtain the first and the second isomorphism theorems. In Section \ref{sec5}, we introduce a notion of Hom-group action and study some properties of Hom-group actions. In the final Section \ref{sec6}, we prove the first Sylow theorem for Hom-groups using Hom-group actions.

Throughout the paper, we assume the map $\al$ is invertible, then some properties can be obtained by Hom-associativity when $\al$ is invertible.

\section{New examples of Hom-groups} \label{sec2}
In this section, we recall the basic notions and properties of Hom-groups and provide some new examples of Hom-groups. 
\begin{definition}{\rm \upcite{Mohammad Hassanzadeh}}
	A Hom-group consists of a set $G$ together with a distinguished member $1\in G$, a bijective set map $\al:G\longrightarrow G$, a binary operation $\mu:G\times G\longrightarrow G$, where these pieces of structure are subject to the following axioms:
	
	{\rm (1)} The product map $\al:G\longrightarrow G$ satisfies the Hom-associativity property$$\mu\left( \al\left( g\right) ,\mu\left( h,k\right) \right)=\mu\left(\mu\left( g,h\right),\al\left( k\right)\right).$$
	For simplicity when there is no confusion, we omit the multiplication sign $\mu$.
	
	{\rm (2)}  The map $\al$ is multiplicative$$\al\left(gh\right)=\al\left( g\right) \al\left( h\right).$$
	
	{\rm (3)}  The element $1$ is called unit and it satisfies the Hom-unitarity conditions $$g1=1g=\al\left( g\right).$$
	
	{\rm (4)}  For every element $g\in G$, there exists an element $g^{-1}\in G$ such that $$gg^{-1}=g^{-1}g=1.$$
	
	$G$ is called a Hom-semigroup if the conditions {\rm (1)},{\rm (2)} are satisfied. A Hom-semigroup satisfying the condition {\rm (3)} is called a Hom-monoid.
	
\end{definition}

\begin{definition}
A Hom-group $(G, \alpha)$ is called an involutive Hom-group if $\alpha^2 = \alpha$.
\end{definition}

\begin{example}\label{group as a Hom-group}
Every group $(G,\mu)$ together with an automorphism $\alpha: G\to G$ has a Hom-group structure $(G, \mu_\alpha,\alpha)$, where $\mu_\alpha=\alpha\circ \mu$. Thus, every group can be thought of as a Hom-group by considering $\alpha$ as the identity map. In this way, the notion of a Hom-group is a generalized notion of a group.
\end{example}
\begin{example}
$(\mathbb{R},+,0,\text{Id})$ is a Hom-group. Note that this example is a special case of the previous example.
\end{example}
\begin{example}
Averaging operation $\oplus$ on $\mathbb{R}$, that is, $a\oplus b=\frac{a+b}{2}$ is a non-associative operation. Therefore, $(\mathbb{R},\oplus)$ is not a group. We show $(\mathbb{R},\oplus,0,\alpha)$ is a Hom-group, where $\alpha: \mathbb{R}\to \mathbb{R}$ is defined as $\alpha(a)=\frac{a}{2}$. It is trivial to check that $\alpha$ is a multiplicative bijective map on $\mathbb{R}$.
\begin{align*}
&\alpha(a)\oplus(b\oplus c)=\frac{a}{2}\oplus \frac{b+c}{2}=\frac{\frac{a}{2}+\frac{b+c}{2}}{2}=\frac{a+(b+c)}{4}\\
&(a\oplus b)\oplus \alpha(c)=\frac{a+b}{2}\oplus \frac{c}{2}=\frac{\frac{a+b}{2}+\frac{c}{2}}{2}=\frac{(a+b)+c}{4}.
\end{align*}
Thus, $\oplus$ is Hom-associative as usual addition $+$ of real numbers is associative.

Now,
\begin{align*}
a\oplus 0=\frac{a}{2}=\alpha(a).
\end{align*}
This implies $0$ is the Hom-identity. Hence, $(\mathbb{R},\oplus,0,\alpha)$ is a Hom-group.
\end{example}
\begin{example}
$a\otimes b=\sqrt[3]{ab}$ is a non-associative operation on $\mathbb{R}-\left\{0\right\}$. Therefore, $\left(\mathbb{R}-\left\{0\right\},1\right)$ is not a group. We show $\left(\mathbb{R}-\left\{0\right\},\otimes,1,\alpha\right)$ is a Hom-group, where $\alpha: \mathbb{R}-\left\{0\right\}\rightarrow\mathbb{R}-\left\{0\right\}$ is defined as $\alpha\left(a\right)=\sqrt[3]{ab}$. It is easy to check that $\alpha$ is a bijective map on $\mathbb{R}-\left\{0\right\}$.

We have
$$\alpha\left(a\otimes b\right)=\alpha\left(\sqrt[3]{ab}\right)=\sqrt[3]{\sqrt[3]{ab}}=\sqrt[9]{ab}$$
$$\alpha\left(a\right)\otimes\alpha\left(b\right)=\sqrt[3]{a}\otimes\sqrt[3]{b}=\sqrt[3]{\sqrt[3]{a}\sqrt[3]{b}}=\sqrt[9]{ab}$$
which imply $\alpha$ is multiplicative.

Moreover,
$$\alpha\left(a\right)\otimes \left(b\otimes c\right)=\sqrt[3]{a}\otimes\left(\sqrt[3]{bc}\right)=\sqrt[3]{\sqrt[3]{a}\sqrt[3]{bc}}=\sqrt[9]{abc}$$
$$\left(a\otimes b\right)\otimes \alpha\left(c\right) =\left(\sqrt[3]{ab}\right)\otimes\sqrt[3]{c}=\sqrt[3]{\sqrt[3]{ab}\sqrt[3]{c}}=\sqrt[9]{abc}$$
$$a\otimes1=\sqrt[3]{a1}=\sqrt[3]{a}=\alpha\left(a\right)=\sqrt[3]{1a}=1\otimes a$$
Therefore, $\otimes$ satisfies Hom-associative and Hom-identity. And it is easy to check that $\frac{1}{a}$ is the inverse of every $a \in \mathbb{R}-\left\{0\right\}$.
\end{example}
\begin{example}
The additive group of integers $(\mathbb{Z}, +)$ has a Hom-group structure with $\oplus$ and $\alpha : \mathbb{Z} \to \mathbb{Z}$ is defined as
$$a \oplus b = - (a+b),~ \alpha(a) = -a.$$
Note that this is an example of an involutive Hom-group.
\end{example}
\begin{example}
Let $X$ be a non-empty set and $\beta$ be a bijective map on $X$. Let $\bf{Sym(X)}$ be the set of all bijective maps on $X$. Then $(\bf{Sym(X)}, \mu, \alpha_{\beta})$ is a Hom-group, where
$$\mu(\sigma_1, \sigma_2)=\beta \circ \sigma_1\circ\beta^{-1}\circ\sigma_2\circ \beta^{-1},~~~\alpha_{\beta}(\sigma)=\beta \circ\sigma\circ \beta^{-1}.$$
It is easy to check that $(\bf{Sym(X)},\mu)$ satisfies Hom-associative condition. Note that $\beta$ is the Hom-identity of $\bf{Sym(X)}$, and the inverse of $\sigma_1$ under this binary operation is $\beta \circ \sigma^{-1}_1 \circ \beta$ . Thus, $(\bf{Sym(X)},\mu,\beta,\alpha_{\beta})$ is a Hom-group.
\end{example}
\begin{example}
This example is a particular case of the above example.
Consider $X=\lbrace 1,2,3\rbrace$, and $\beta = (1~2)$. We have
$$\bf{Sym(X)}=\lbrace id, (1~2), (1~3), (2~3), (1~2~3), (1~3~2)\rbrace$$
Using the binary operation $\mu$ on $\bf{Sym(X)}$, we get the following table:
\begin{equation*}
\begin{array}{|c|c|c|c|c|c|c|}
\hline
\ \mu  & id & (1~2) & (1~3)& (2~3)&(1~2~3)&(1~3~2)\\ \hline
id &(1~2) & id &(1~2~3) & (1~3~2) &(1~3)&(2~3)\\ \hline
(1~2) & id &(1~2)&(2~3) & (1~3) &(1~3~2) &(1~2~3)\\ \hline
(1~3) & (1~3~2) & (2~3)& (1~3) & (1~2) &(1~2~3)& id \\ \hline
(2~3) & (1~2~3)&(1~3) & (1~2) & (2~3) & id & (1~3~2) \\ \hline
(1~2~3)&(2~3)& (1~3~2)& id & (1~2~3)&(1~2)&(1~3) \\  \hline
(1~3~2)& (1~3)& (1~2~3)&(1~3~2)& id &(2~3)&(1~2)\\  \hline
\end{array}%
\end{equation*}%
Table for Hom-map $\alpha_\beta$ is given by:
\begin{equation*}
\begin{array}{|c|c|}
\hline
\ id  & id\\ \hline
(1~2)&(1~2)\\ \hline
(1~3) & (2~3)\\ \hline
(2~3) & (1~3)\\ \hline
(1~2~3)&(1~3~2)\\  \hline
(1~3~2)& (1~2~3)\\  \hline
\end{array}%
\end{equation*}%
\end{example}

\section{Fundamental properties of Hom-groups} \label{sec2.5}
In this section, we introduce several new notions of Hom-groups such as the law of cancellation and equivalent descriptions of Hom-groups. 
\begin{proposition}
	If $G$ is a Hom-monoid, then the unit is unique. Therefore, there is a unique unit in Hom-groups.	
\end{proposition}

\begin{proof}
	Let $1'$ be another unit of $G$, then $1'=\al\left( 1'\right)=1'1=\al\left( 1\right)=1.$
\end{proof}

\begin{proposition}
	Let $G$ be a Hom-group. Then
	
	{\rm (1)}{\rm \upcite{JMS}} $\alpha(1)=1$.
	
	{\rm (1)}{\rm \upcite{Mohammad Hassanzadeh}} For each $g\in G$, the inverse element $g^{-1}$ is unique.
	
	{\rm (2)}{\rm \upcite{MH}} For each $g\in G$, $\left( g^{-1}\right) ^{-1}=g$.
	
	{\rm (3)}{\rm \upcite{Mohammad Hassanzadeh}} For $g,h\in G$, $\left( gh\right)^{-1}=h^{-1}g^{-1}$.
	
	{\rm (4)}{\rm \upcite{MH}} For each $g\in G$, $\left(\al(g)\right)^{-1}=\al(g^{-1})$.
\end{proposition}

\begin{proposition}
      Let $G$ be a Hom-group. If there is an element $g\in G$ such that $gg=\al\left(g\right)=g1$, then $g=1$.
\end{proposition}

\begin{proof}
	Assume that $g\in G$ satisfies $gg=\al\left(g\right)=g1$. Then
	\begin{eqnarray*}
		\left( gg\right) \al\left(g^{-1}\right) =\al\left( g\right) \al\left( g^{-1}\right) =\al\left( gg^{-1}\right) =\al\left( 1\right) =1,\\
		\left( gg\right) \al\left(g^{-1}\right) =\al\left( g\right) \left( gg^{-1}\right) =\al\left( g\right) 1=\al^{2}\left( g\right),
	\end{eqnarray*}
	which imply that $\al^{2}\left( g\right)=1$. Hence $g=\al^{-2}\left( 1\right) =1$ by $\al$ being invertible.
\end{proof}

\begin{proposition}
	If the elements $g,h,k$ in a Hom-group satisfy $gh=gk$ or $hg=kg$, then $h=k$.
\end{proposition}

\begin{proof}
		Assume that $gh=gk$. By left multiplying $\al\left(g^{-1}\right)$ to both sides of $gh=gk$ the equality $\al\left(g^{-1}\right)\left(gh\right)=\al\left(g^{-1}\right)\left(gk\right)$ is obtained, so $\left(g^{-1}g\right)\al\left(h\right)=\left(g^{-1}g\right)\al\left(k\right)$ by the law of Hom-associativity. Hence $1\al\left(h\right)=1\al\left(k\right)$	by the definition of Hom-invertibility. By Hom-unitarity we have $\al^{2}\left(h\right)=\al^{2}\left(k\right)$. Then $h=k$, since $\al$ is invertible. The same argument shows that $hg=kh$ implies $h=k$.
\end{proof}

\begin{proposition}
	Let $G$ be a Hom-group. Then
	
	{\rm (1)} For all $g,h\in G$, $\al^{-1}\left(gh\right)=\al^{-1}\left(g\right)\al^{-1}\left(h\right)$.
	
	{\rm (2)} For all $g,h,k\in G$, $\left(\al^{-1}\left(g\right)h\right)k=g\left(h\al^{-1}\left(k\right)\right)$.
\end{proposition}

\begin{proof}
	{\rm (1)} For all $g,h\in G$,
	\begin{align*}
	\al\left(\al^{-1}\left(gh\right)\right)=gh=\al\left(\al^{-1}\left(g\right) \right)\al\left(\al^{-1}\left (h\right)\right)=\al\left(\al^{-1}\left( g\right)\al^{-1}\left (h\right)\right).
	\end{align*}
	We have  $\al^{-1}\left(gh\right)=\al^{-1}\left(g\right)\al^{-1}\left(h\right)$, since $\al$ is invertible.

    {\rm (2)} For all $g,h,k\in G$, since $\al$ is invertible, there exists $h'$ such that $\al^{-1}\left(h'\right)=h$, and
    \begin{gather*}
    \al\left(g\right)\left(h'k\right)=\left(gh'\right)\al\left(k\right),\\
    \al^{-1}\left(\al\left(g\right)\left(h'k\right)\right)=\al^{-1}\left(\left(gh'\right)\al\left(k\right)\right),\\
    \al^{-1}\left(\al\left(g\right)\right)\al^{-1}\left(h'k\right)=\al^{-1}\left(gh'\right)\al^{-1}\left(\al\left(k\right)\right),\\
    g\al^{-1}\left(h'k\right)=\al^{-1}\left(gh'\right)k,\\
    g\left(\al^{-1}\left(h'\right)\al^{-1}\left(k\right)\right)=\left(\al^{-1}\left(g\right)\al^{-1}\left(h'\right)\right)k,
    g\left(h\al^{-1}\left(k\right)\right)=\left(\al^{-1}\left(g\right)h\right)k,
    \end{gather*}
    which complete the proof.
\end{proof}

\begin{theorem}
	Let $G$ be a Hom-semigroup. Then $G$ is a Hom-group if and only if the following two conditions hold:
	
	{\rm (1)} There exists an element $1\in G$ such that $1g=\al\left(g\right)$, $\al\left(1\right)=1$ (left unit).
	
	{\rm (2)} For each $g\in G$, there exists an element $g^{-1}\in G$ such that $g^{-1}g=1$ (left inverse).
\end{theorem}

\begin{proof}
	$\left(\Longrightarrow\right)$ Trivial.
	
	$\left(\Longleftarrow\right)$
	We need to prove: $1$ is the right unit, and $g^{-1}$ is the right inverse, i.e. $gg^{-1}=1$.
	Let $h=gg^{-1}$, then
\begin{align*}
	hh&=\left(gg^{-1}\right) \left(gg^{-1}\right)\\
	&=\left(gg^{-1}\right)\al\left(\al^{-1}\left(gg^{-1}\right)\right)\\
	&=\al\left(g\right)\left(g^{-1}\left(\al^{-1}\left(gg^{-1}\right)\right)\right)\\
	&=\al\left(g\right)\left(g^{-1}\left(\al^{-1}\left(g\right)\al^{-1}\left(g^{-1}\right)\right)\right)\\
	&=\al\left(g\right)\left(\al^{-1}\left(g^{-1}g\right)g^{-1}\right)\\
	&=\al\left(g\right)\left(\al^{-1}\left(1\right)g^{-1}\right)\\
	&=\al\left(g\right)\left(1g^{-1}\right)\\
	&=\al\left(g\right)\al\left(g^{-1}\right)\\
	&=\al\left(gg^{-1}\right)\\
	&=\al\left(h\right).
\end{align*}

Moreover, we have $h=\al^{-1}(hh)=\al^{-1}(h)\al^{-1}(h)$ and $\al\left(1\right)=1=h^{-1}h=h^{-1}\left(\al^{-1}(h)\al^{-1}(h)\right)=\left(\al^{-1}(h)\al^{-1}(h)\right)h=\al^{-1}\left(1\right)h=1h=\al\left(h\right)=\al\left(gg^{-1}\right)$.
$1=gg^{-1}$ since $\al$ is invertible. Thus $g^{-1}$ is the right inverse.

In addition, $g1=g\al^{-1}\left(1\right)=g\al\left(g^{-1}g\right)=g\left(\al^{-1}(g^{-1})\al^{-1}(g)\right)=\left(\al^{-1}(g)\al^{-1}(g^{-1})\right)g=\al^{-1}\left(gg^{-1}\right)g$ $=\al^{-1}\left(1\right)g=1g=\al\left(g\right)$. Thus $1$ is the right unit.

This shows that $G$ is a Hom-group.
\end{proof}

\begin{proposition}
	Let $G$ be a nonempty Hom-semigroup with $|G|<+\infty$. Then $G$ is a Hom-group if and only if $G$ satisfies the law of left and right cancellation.
\end{proposition}

\begin{proof}
	Let $G=\left\{ a_1,a_2,\cdots,a_n \right\}$. Since $\al$ is invertible, we have
	$$G=\left\{a_1,a_1,\cdots,a_n\right\}
	=\left\{\al\left(a_{1}\right),\al\left(a_{2}\right),\cdots,\al\left(a_{n}\right)\right\}.$$
    For any $k\in \left\{1,2,\cdots,n\right\}$, by right multiplying $a_k$ to every element of $G$,  $G'=\left\{ a_1a_k,a_2a_k,\cdots,a_na_k \right\}$ is obtained.

    We claim that $G=G'$. By the closure of binary operation, we have $G'\subseteq G$. We need to show that $|G'|=n=|G|$, otherwise there exist $1\le i,j\le n$ such that $a_ia_k=a_ja_k$. $a_i=a_j$ by the law of right cancellation, which leads to a contraction. Thus $|G'|=n=|G|$. The claim is proved.

    There exists $r\in \left\{1,2,\cdots,n\right\}$ such that $\al\left(a_k\right)=a_ra_k$ by $G=G'$. So $\al\left(a_k\right)\al\left(a_k\right)=\al\left(a_ka_k\right)=\al\left(a_k\right)\al\left(a_k\right)=\al\left(a_k\right)a_ra_k=\left(a_ka_r\right)\al\left(a_k\right)$. By the law of right cancelation, we get $\al\left(a_k\right)=a_ka_r$. We claim that $a_r$ is the left unit of $G$. For any $\al\left(a_i\right)\in G$, $\al\left(a_k\right)\al\left(a_i\right)=\left(a_ka_r\right)\al\left(a_i\right)=\al\left(a_k\right)\left(a_ra_i\right)$. By the law of left cancellation, we get $\al\left(a_i\right)=a_ra_i$. In addition, $\al\left(a_r\right)=a_r$, otherwise assume that $\al\left(a_r\right)=a_s\ne a_r$, we have $a_r\left(a_ra_s\right)=\al\left(a_ra_s\right)=\al\left(a_r\right)\al\left(a_s\right)=a_s\left(a_ra_s\right)$. By the law of right cancelation, we get $a_r=a_s$, leading to contradiction. Therefore $a_r$ is the left unit of $G$. This proves the claim.

    Since $G'=G$, there exists $j\in \left\{ 1,2,\cdots,n \right\}$ such that $a_r=a_ja_k$, that is, $a_j$ is the left inverse of $a_k$ with respect to the left unit $a_r$ in $G$. Therefore, $G$ is a Hom-group.
\end{proof}

\begin{definition}{\rm \upcite{MH}}
	Let $\left(G,\al\right)$ be a Hom-group and $H$ a nonempty subset that is closed under the product in $G$. If $\left(H,\al\right)$ is itself a Hom-group under the product in $G$, then $H$ is said to be a Hom-subgroup of $G$, denoted by $H\preceq G$.
\end{definition}

\begin{proposition}\label{eq1}
	Let $H$ be a Hom-subgroup of $G$. Then
	
	{\rm (1)} $1_H=1_G$.
	
	{\rm (2)} For each $g\in H$, $g^{-1}_H=g^{-1}_G$.
\end{proposition}

\begin{proof}
	{\rm (1)} $1_H1_H=\al\left(1_H\right)=1_H=\al\left(1_H\right)=1_G1_H$, so $1_H=1_G$ by the law of cancellation in $G$.
	
	{\rm (2)} For each $g\in H$, we have $gg^{-1}_H=1_H=1_G=gg^{-1}_G$. Hence $g^{-1}_H=g^{-1}_G$  by the law of cancellation in $G$.
\end{proof}

\begin{theorem}
	Let $H$, $K$ be two Hom-subgroups of $G$. Then
	
	{\rm (1)} $H\cap K\preceq G$;
	
	{\rm (2)} $H\cup K\preceq G$ if and only if $H\subset K$ or $K\subset H$.
\end{theorem}

\begin{proof}
	{\rm (1)} Obviously, we have $1_G\in H\cap K\not= \emptyset$ by Proposition \ref{eq1}. For any $g,h\in H\cap K$, we have $g,h\in H$ and $g,h\in K$. So we have $gh^{-1}\in H$ and $gh^{-1}\in K$ by $H\preceq G$ and $K\preceq G$. Therefore, $gh^{-1}\in H\cap K$, which implies $H\cap K\preceq G$.
	
	{\rm (2)} $\left(\Longleftarrow\right)$ Assume that $H\subset K$ or $K\subset H$, we have $H\cup K=K$ or $H$, then $H\cup K\preceq G$.
	
	$\left(\Longrightarrow\right)$ Suppose that $H\cup K\preceq G$, we need to prove $H\subset K$ or $K\subset H$. Suppose that $H\subsetneq K$ and $K\subsetneq H$, then there exists $g\in H$ while $g\not\in K$, and there exists $h\in K$ while $h\not\in H$. Hence $g\in H\cup K$ and $h\in H\cup K$. Then $gh\in H\cup K$ follows from $H\cup K\preceq G$. Thus $gh\in H$ or $gh\in K$. Suppose that $gh\in H$. Since $H\preceq G$, we have $\al\left(g^{-1}\right)\in H$. The equality $\al^{2}\left(h\right)=1\al\left(h\right)=\left(g^{-1}g\right)\al\left(h\right)=\al\left(g^{-1}\right)\left(gh\right)$ implies that $\al^{2}\left(h\right)\in H$. Hence $h\in H$, which contradicts the assumption that $h\not\in H$. Similarly, if $gh\in K$, we have $g\in K$, which contradicts the assumption that $g\not\in K$. This shows that $H\subset K$ or $K\subset H$.
\end{proof}

\section{Hom-normal subgroups and Hom-quotient groups}\label{sec3}

In this section, we recall the definitions of a coset of Hom-groups. Then we focus on Hom-normal subgroups, which is the foundation of Hom-quotient groups. Finally, we set up Hom-quotient groups and prove that it is a Hom-group.

\begin{definition}{\rm \upcite{Mohammad Hassanzadeh}}
	Let $H$ be a Hom-subgroup of a Hom-group $G$. Let $a$ be an element of $G$, Denote $aH=\left\{ah|h\in H\right\}$ and $Ha=\left\{ha|h\in H\right\}$. The sets $aH$ and $Ha$ are called a Hom-left coset and a Hom-right coset of $H$ in $G$ respectively.
\end{definition}

\begin{proposition}\label{th-abcd}
	Let $H$ be a Hom-subgroup of a finite Hom-group $G$.
    For all $a,b\in G$ the following statements are equivalent:
	
	{\rm (1)} $aH=bH$.
	
	{\rm (2)} $aH\cap bH\not= \emptyset$.
	
	{\rm (3)} $a^{-1}b\in H$.
	
	{\rm (4)} $\al\left(b\right)\in aH$.
	
	{\rm (5)} $\al\left(a\right)H=\al\left(b\right)H$.
\end{proposition}

\begin{proof}
$\left(1\right)\Longrightarrow\left(2\right)$ It\rq s obvious that $aH\cap bH\neq \emptyset$.
	
	$\left(2\right)\Longrightarrow\left(3\right)$
	There exist $h_1,h_2\in H$ such that $ah_1=bh_2$. By multiplying $\al\left(a^{-1}\right)$ to both sides of $ah_1=bh_2$ the equality $\al\left(a^{-1}\right)\left(bh_2\right)=\al\left(a^{-1}\right)\left(ah_1\right)$ is obtained. Then $$\left(a^{-1}b\right)\al\left(h_2\right)=\left(a^{-1}a\right)\al\left(h_1\right)=1\al\left(h_1\right)=\al^{2}\left(h_1\right).$$
	By multiplying $\al^{2}\left(h_2^{-1}\right)$ to both sides of $\left(a^{-1}b\right)\al\left(h_2\right)=\al^{2}\left(h_1\right)$, we have $\left(\left(a^{-1}b\right)\al\left(h_2\right)\right)\al^{2}\left(h_2^{-1}\right)=\al^{2}\left(h_1\right)\al^{2}\left(h_2^{-1}\right)$. Then
	\begin{align*}
	\al^{2}\left(h_1h_2^{-1}\right)&=\al\left(a^{-1}b\right)\left(\al\left(h_2\right)\al\left(h_2^{-1}\right)\right)\\
	&=\al\left(a^{-1}b\right)\al\left(h_2h_2^{-1}\right)\\
	&=\al\left(a^{-1}b\right)\al\left(1\right)\\
	&=\al^{2}\left(a^{-1}b\right).
	\end{align*}
	Hence $a^{-1}b=h_1h_2^{-1}\in H$ by $\al$ being invertible.
	
	$\left(3\right)\Longrightarrow\left(4\right)$
	Since $\al\left(b\right)=1b=\al^{-1}\left(1\right)b=\al^{-1}\left(aa^{-1}\right)b=a\left(\al^{-1}\left(a^{-1}\right)\al^{-1}\left(b\right)\right)=a\al^{-1}\left(a^{-1}b\right)$ and $a^{-1}b\in H$, we have $\al\left(b\right)=a\al^{-1}\left(a^{-1}b\right)\in aH$ by the fact that $\al$ is bijective.

    $\left(4\right)\Longrightarrow\left(5\right)$
    Let $\al\left(b\right)\in aH$. Then there exists $h_1\in H$ such that $\al\left(b\right)=ah_1$. For each $\al\left(b\right)h\in \al\left(b\right)H$, $\al\left(b\right)h=\left(ah_1\right)h$. Since there exists $h_2\in H$ such that $\al\left(h_2\right)=h$, we have $\al\left(b\right)h=\left(ah_1\right)\al\left(h_2\right)=\al\left(a\right)\left(h_1h_2\right)\in \al\left(a\right)H$. Thus $\al\left(b\right)H\subseteq\al\left(a\right)H$.
	
	In addition, by multiplying $\al\left(h_1^{-1}\right)$ to both sides of $\al\left(b\right)=ah_1$ the equality $\al\left(b\right)\al\left(h_1^{-1}\right)=\left(ah_1\right)\al\left(h_1^{-1}\right)$ is obtained, so $\al\left(bh_1^{-1}\right)=\al\left(a\right)\left(h_1h_1^{-1}\right)=\al\left(a\right)1=\al^{2}\left(a\right)$. Since $\al$ is invertible, we have $bh_1^{-1}=\al\left(a\right)$. And there exists $h_3\in H$ such that $\al\left(h_3\right)=h$, then $\al\left(a\right)h=\left(bh_1^{-1}\right)\al\left(h_3\right)=\al\left(b\right)\left(h_1^{-1}h_3\right)\in \al\left(b\right)H$. Thus $\al\left(a\right)H\subseteq\al\left(b\right)H$.
	
	Hence, $\al\left(a\right)H=\al\left(b\right)H$.
	
	$\left(5\right)\Longrightarrow\left(1\right)$
	For each $ah\in aH$, there exists $h_1\in H$ such that $\al\left(h\right)=h_1$. Then $\al\left(ah\right)=\al\left(a\right)\al\left(h\right)=\al\left(a\right)h_1$, which implies that $ah=\al^{-1}\left(\al\left(b\right)h_1\right)=b\al^{-1}\left(h_1\right)\in bH$. This shows that $aH\subseteq bH$. Similarlily, we have $bH\subseteq aH$. Therefore, we get $aH=bH$.
	
	So the proposition holds.
\end{proof}

\begin{definition}
	Let $H$ be a Hom-subgroup of a Hom-group $G$. If $gH=Hg$ for any $g\in G$, then $H$ is called a Hom-normal subgroup of $G$, denoted by $H\triangleleft G$.
\end{definition}

\begin{proposition}
		Let $H$ be a Hom-subgroup of a Hom-group $G$, for all $g\in G$ the following statements are equivalent:
	
	{\rm (1)} $gH=Hg$.
	
	{\rm (2)} For each $h\in H,\left(gh\right)\al\left(g^{-1}\right)\in H$.
	
	{\rm (3)} $\left(gH\right)\al\left(g^{-1}\right)\subseteq H$.
	
	{\rm (4)} $\left(gH\right)\al\left(g^{-1}\right)=H$.
	
	{\rm (5)} $\al\left(g\right)H=H\al\left(g\right)$.
	
\end{proposition}

\begin{proof}
	$\left(1\right)\Longrightarrow\left(2\right)$
	Let $gH=Hg$. Then there exist $h_1,h_2\in H$ such that $gh_1=h_2g$, hence we have $\left(gh_1\right)\al\left(g^{-1}\right)=\left(h_2g\right)\al\left(g^{-1}\right)=\al\left(h_2\right)\left(gg^{-1}\right)=\al\left(h_2\right)1=\al^{2}\left(h_1\right)\in H$.
	
	$\left(2\right)\Longrightarrow\left(3\right)$
	It is obvious that $\left(gH\right)\al\left(g^{-1}\right)\subseteq H$.
	
	$\left(3\right)\Longrightarrow\left(4\right)$
	We only need to show that $H\subseteq \left(gH\right)\al\left(g^{-1}\right)$. For every $h\in H$, there exists $h'$ such that
	\begin{align*}
	h&=\al^{2}\left(h'\right)\\
	&=\al\left(h'\right)1\\
	&=\al\left(h'\right)\left(gg^{-1}\right)\\
	&=\left(h'g\right)\al\left(g^{-1}\right).
	\end{align*}
	Note that
	\begin{align*}
	\left(\al\left(\al^{-1}\left(h'\right)\right)g\right)
	&=\left(\left(1\al^{-1}\left(h'\right)\right)g\right)\\
	&=\left(\left(\al^{-2}\left(gg^{-1}\right)\al^{-1}\left(h'\right)\right)g\right)\\
	&=\left(\left(\left(\al^{-2}\left(g\right)\al^{-2}\left(g^{-1}\right)\right)\al^{-1}\left(h'\right)\right)g\right)\\
	&=\left(\left(\al^{-1}\left(g\right)\left(\al^{-2}\left(g^{-1}\right)\al^{-2}\left(h'\right)\right)\right)g\right)\\
	&=\left(g\left(\left(\al^{-2}\left(g^{-1}\right)\al^{-2}\left(h'\right)\right)\al^{-1}\left(g\right)\right)\right).
	\end{align*}
	Due to $g^{-1}\in G$ and $\al$ being bijective, we have $\left( g^{-1}h'\right)\al\left(g\right)\in H$ and $ \al^{-2}\left(g^{-1}h'\right)\al^{-1}\left(g\right)\in H$ which implies that $\left(\al^{-2}\left(g^{-1}\right)\al^{-2}\left(h'\right)\right)\al^{-1}\left(g\right)\in H$. Then we have $h\in gH\al\left(g^{-1}\right)$ and $H\subseteq \left( gH\right)\al\left(g^{-1}\right)$.
	Therefore, $\left( gH\right)\al\left(g^{-1}\right)=H$ holds.
	
	$\left(4\right)\Longrightarrow\left(5\right)$
	For each $h\al\left(g\right)\in H\al\left(g\right)$, there exists $h_1$ such that $\al^{-1}\left(h_1\right)=h$, and there exists $h_2$ such that $h_1=\left(gh_2\right)\al\left(g^{-1}\right)$. Then we have $h\al\left(g\right)=\al^{-1}\left(h_1\right)\al\left(g\right)=\al^{-1}\left(gh_2\al\left(g^{-1}\right)\right)\al\left(g\right)=\al\left(\al^{-1}\left(gh_2\right)\right)\left(gg^{-1}\right)=\left(gh_2\right)1=\al\left(gh_2\right)=\al\left(g\right)\al\left(h_2\right)\in \al\left(g\right)H$. Thus $H\al\left(g\right)\subseteq\al\left(g\right)H$.
	
	Similarly, for each $\al\left(g\right)h\in \al\left(g\right)H$, there exists $h_3$ such that $\al^{-1}\left(h_3\right)=h$, and there exists $h_4$ such that $h_3=\left(g^{-1}h_4\right)\al\left(g\right)$. Then we have $\al\left(g\right)h=\al\left(g\right)\al^{-1}\left(h_3\right)=\al\left(g\right)\al^{-1}\left(\left(g^{-1}h_4\right)\al\left(g\right)\right)=\left(g\al^{-1}\left(g^{-1}h_4\right)\right)\al\left(g\right)=\left(\al^{-1}\left(gg^{-1}\right)h_4\right)\al\left(g\right)=\left(\al^{-1}\left(1\right)h_4\right)\al\left(g\right)=\left(1h_4\right)\al\left(g\right)=\al\left(h_4\right)\al\left(g\right)\in H\al\left(g\right)$. Thus $\al\left(g\right)H\subseteq H\al\left(g\right)$.
	
	Hence, $\al\left(g\right)H=H\al\left(g\right)$.
	
	$\left(5\right)\Longrightarrow\left(1\right)$ For every $g$, there exists $g'\in G$ such that $\al\left(g'\right)=g$. According to $\left(5\right)$, we have $\al\left(g'\right)H=H\al\left(g'\right)$. Thus, $gH=Hg$.
	
	So the proposition holds.
\end{proof}

\begin{example}
	Let $G$ be a Hom-group, then $Z\left(G\right)\triangleleft G$, where $Z\left(G\right)=\left\{h\in G|gh=hg, \forall g\in G\right\}$.
\end{example}

\begin{proof}
	By reference {\rm \cite{Mohammad Hassanzadeh}}, $Z\left(G\right)$ is a Hom-subgroup of $G$. For every $h\in Z\left(G\right), g\in G$, we have $\left(gh\right)\al\left(g^{-1}\right)=\left(hg\right)\al\left(g^{-1}\right)=\al\left(h\right)\left(gg^{-1}\right)=\al\left(h\right)1=\al^{2}\left(h\right)\in Z\left(G\right)$, which implies $Z\left(G\right)\triangleleft G$.
\end{proof}

\begin{proposition}
	If $H$ is a Hom-normal subgroup of a Hom-group $G$ and $G/H$ is the set of all (left) Hom-cosets of $H$ in $G$, then $G/H$ is a Hom-group under the binary operation given by $\widetilde{\mu}\left(aH,bH\right) =\mu\left(a,b\right)H$ with $\widetilde{\al}\left(aH\right)=\alpha\left(a\right)H$.
\end{proposition}

\begin{proof}First, we verify the two definitions are well-defined.

Verify $\widetilde{\mu}: G/H\times G/H\longrightarrow G/H$ is well-defined:
Assume that $\overline{a_{1}}=\overline{a_{2}}$ and $\overline{b_{1}}=\overline{b_{2}}$, we need to prove $\overline{a_{1}} \overline{b_{1}}=\overline{a_{2}} \overline{b_{2}}$, that is $\overline{a_{1} b_{1}}=\overline{a_{2} b_{2}}$. Since
	\begin{eqnarray*}
		\left(a_{1} b_{1}\right)^{-1} \left(a_{2} b_{2}\right)&=& \left(b_{1}^{-1}  a_{1}^{-1}\right) \left(a_{2} b_{2}\right)\\
		&=& \al\left(b_{1}^{-1}\right)\left(a_{1}^{-1}  \al^{-1}\left(a_{2} b_{2}\right)\right)\\
		&=& \al\left(b_{1}^{-1}\right)\left(\left(\al^{-1}\left(a_{1}^{-1}\right) \al^{-1}\left(a_{2}\right)\right) \left(\alpha \al^{-1}\left(b_{2}\right)\right)\right)\\
		&=& \al\left(b_{1}^{-1}\right) \left(\al^{-1}\left(a_{1}^{-1} a_{2}\right) b_{2}\right),
	\end{eqnarray*}
	and $ \overline{a}_{1}=\overline{a}_{2}$ which implies that $a_{1}H=a_{2}H$, there exists $h\in H$ such that $a_{1}^{-1}a_{2}=h$. Then we get $\al\left(b_{1}^{-1}\right) \left(\al^{-1}\left(a_{1}^{-1} a_{2}\right) b_{2}\right) =\al\left(b_{1}^{-1}\right)\left(\alpha^{-1}(h)  b_{2}\right)=\al\left(b_{1}^{-1}\right) \left(b_{2} h^{\prime}\right)=\left(b_{1}^{-1} b_{2}\right)  \al\left(h^{\prime}\right) \in H$, that is $\left(a_{1} b_{1}\right) H=\left(a_{2} b_{2}\right) H$. Therefore, $\widetilde{\mu}$ is well-defined.
		
Verify $\widetilde{\al}: G/H\longrightarrow G/H$ is well-defined: Assume that $aH=bH$, we have $\al\left(a\right)H=\al\left(b\right)H$ by proposition \ref{th-abcd}. Therefore, $\widetilde{\alpha}$ is well-defined.

    Now, we prove $G/H$ is a Hom-group:

    (1) $\tilde{\al}(\overline{g})(\overline{h} \cdot\overline{k})=\widetilde{\al}(g H)(h H k H)=\tilde{\al}(g H)(h k) H=\al(g) H (h k) H=\left(\al(g)(h k)\right) H=\left((g h) \al(k)\right)H=(g h) H  \al(k) H=(g H  k H) \al(k) H=(g H k H)  \widetilde{\al}(k H)$, then the Hom-associativity property is obtained.
	
    (2) It is obvious that  $\widetilde{\al}$ is a bijective set map with $\al$ bijective. By $\widetilde{\al}\left(\left(g H\right)\left( k H\right)\right)=\widetilde{\al}\left(gkH\right)=\al(g k)H=\left(\al(g)\al(k)\right)H=\left(\al(g)H\right)\left(\al(k) H\right)=\widetilde{\al}\left(gH\right)\widetilde{\al}\left(kH\right)$, $\widetilde{\al}$ is multiplicative.
	
	(3) $H$ is the identity element: since $(gH)H=(gH) (1_{H}H)=(g1_{H})H=\al(g)H=\widetilde{\al}(gH)$, we have $(gH)H=H(gH)=\widetilde{\alpha}\left( g H\right) $.
	
	(4) For all $gH$, there exists $ g^{-1}H$ such that $(gH)(g^{-1}H)=(gg^{-1})H=1_{H}H=H$.
	
	This shows that $G/H$ is a Hom-group.
\end{proof}

\section{Homomorphisms of Hom-groups}\label{sec4}

In this section, we introduce the main results of this paper, about homomorphisms of Hom-groups. One of the most important theorems is the fundamental theorem of homomorphisms of Hom-groups. We also obtain a great many properties of homomorphisms of Hom-groups.

\begin{definition}{\rm \upcite{Mohammad Hassanzadeh}}
	Let $\left(G,\al\right)$, $\left(H,\beta\right)$ be two Hom-groups. The map $f:G\longrightarrow H$ is called a homomorphism of Hom-groups, if $f$ satisfies the following two conditions:
	
{\rm (1)} For all $g,k\in G $, $f\left( gk\right)=f\left( g\right)f\left( k\right)$.

{\rm (2)} For each $g\in G$, $\beta \left( f\left( g\right) \right)=f\left( \alpha \left( g\right) \right)$.

Moreover, if $f:G\longrightarrow H$ is bijective, then we call $f$ isomorphic.
\end{definition}

\begin{definition}
 A map $\phi: G \to H$ is called a weak homomorphism of Hom-groups if $\phi (1_G) = 1_H$ and $\beta \circ \phi (g_1 g_2) = (\phi \circ \alpha (g_1))(\phi \circ \alpha (g_2))$.
\end{definition}

\begin{proposition}
	Let $f:\left(G,\alpha\right)\longrightarrow \left(H,\beta\right)$ be a homomorphism of Hom-groups. Then
	
	{\rm (1)} $f\left( 1_{G}\right) =1_{H}$.
	
	{\rm (2)} For each $g\in G$, $f\left(g^{-1}\right)=\left(f\left(g\right)\right)^{-1}$.
\end{proposition}

\begin{proof}
	{\rm (1)} According to  $f\left(1_G\right)f\left(1_G\right)=f\left(1_G1_G\right)=f\left(\al\left(1_G\right)\right)=\beta\left(f\left(1_G\right)\right)=1_Hf\left(1_G\right)$ and the law of right cancellation, we have $f\left( 1_{G}\right) =1_{H}$.
	
	{\rm (2)} For each $g\in G$, we have  $$f\left(g^{-1}\right)f\left(g\right)=f\left(g^{-1}g\right)=f\left(1_G\right)=1_H,$$ $$f\left(g\right)f\left(g^{-1}\right)=f\left(gg^{-1}\right)=f\left(1_G\right)=1_H,$$ which imply that $f\left(g^{-1}\right)=\left(f\left(g\right)\right)^{-1}$.
\end{proof}

\begin{proposition}\label{eq2}
	Let $f:\left(G,\alpha\right)\longrightarrow \left(H,\beta\right)$ be a homomorphism of Hom-groups. Then $Kerf=\lbrace g\in G|f\left(g\right)=1_{H}\rbrace$ is a Hom-normal subgroup of $G$.
\end{proposition}

\begin{proof}
	By reference \cite[Lemma 2.11]{MH}, we know $Kerf$ is a Hom-subgroup of $G$, now we need to prove $Kerf$ is a Hom-normal subgroup of $G$, that is for every $g\in G$, $gKerf=Kerfg$.

	We only need to prove for all $h\in Kerf$, $g\in G$, we have $\left( gh\right) \alpha\left( g^{-1}\right) \in Kerf$.

	Since $f \left( \left( gh\right) \al \left( g^{-1}\right) \right) = f \left( gh\right)\beta \left( f \left( g^{-1}\right)  \right) =\left( f\left( g\right) f\left( h\right) \right)\beta \left( f\left( g^{-1}\right) \right) =\left( f\left( g\right) 1_{H}\right) \beta\left( f\left( g^{-1}\right) \right) = \beta(f(g)) \beta(f(g^{-1})) = \beta (f(g g^{-1}))=\beta(f(1_{G}))=\beta(1_{H})=1_{H} $, we have $(gh) \alpha(g^{-1})\in Kerf$, which completes the proof.
\end{proof}

\begin{proposition}
    Let $f:\left(G,\alpha\right)\longrightarrow \left(H,\beta\right)$ be a homomorphism of Hom-groups, $Imf=\lbrace f\left(g\right)|g\in G\rbrace$ is a Hom-subgroup of $H$.
\end{proposition}

\begin{proof}
    Since $\beta\left(f\left(x\right)\right)=f\left(\alpha\left(x\right)\right)\in f\left(G\right)$, we have $\beta$ is a bijective set map from $f\left(G\right)$ to $f\left(G\right)$. According to $f\left(g\right) f\left(g^{\prime}\right)=f\left(gg^{\prime}\right)\in f\left(G\right)$, the closure of binary operation in $f\left(G\right)$ is obtained.
	Because $\left(H,\beta\right)$ is a Hom-group, we have $\beta$ and the operation of $f(g)$  satisfy four conditions of definition, which implies that $Imf=\lbrace f\left(g\right)|g\in G\rbrace$ is a Hom-subgroup of $H$.
\end{proof}

\begin{proposition}
	Let $f:\left(G,\alpha\right)\longrightarrow \left(H,\beta\right)$ be a homomorphism of Hom-groups. If $G'$ is a Hom-subgroup of $G$ and $H'$ is a Hom-subgroup of $H$, then
	
	{\rm (1)} $f\left(G'\right)$ is a Hom-subgroup of $H$.
	
	{\rm (2)} $f^{-1}\left(H'\right)$ is a Hom-subgroup of $G$.
\end{proposition}

\begin{proof}
	Obviously, $f\left(G'\right),f^{-1}\left(H'\right)\not= \emptyset$.
	
	{\rm (1)} For all $g,h\in f\left(G'\right)$, they can be written as $f\left(a\right)$, $f\left(b\right)$, where $a,b\in G'$. Then we have $$f\left(a\right)\left(f\left(b\right)\right)^{-1}=f\left(a\right)f\left(b^{-1}\right)=f\left(ab^{-1}\right)\in f\left(G'\right).$$ Thus, $f\left(G'\right)$ is a Hom-subgroup of $H$.
	
	{\rm (2)} For all $ a,b\in f^{-1}\left(H'\right)$, we have $f\left(a\right),f\left(b\right)\in H'$. Since $H'$ is a Hom-subgroup of $H$, we know $$f\left(ab^{-1}\right)=f\left(a\right)f\left(b^{-1}\right)=f\left(a\right)\left(f\left(b\right)\right)^{-1}\in H',$$ which implies that $ab^{-1}\in f^{-1}\left(H'\right)$. Therefore, $f^{-1}\left(H'\right)$ is a Hom-subgroup of $G$.
\end{proof}

\begin{proposition}
	Let $f:\left(G,\alpha\right)\longrightarrow \left(H,\beta\right)$ be a homomorphism of Hom-groups, then
	
	{\rm (1)} $f$ is a monomorphism $\Longleftrightarrow Kerf=\lbrace 1_{G}\rbrace$.
	
	{\rm (2)} $f$ is an epimorphism $\Longleftrightarrow Imf=H$.
	
	{\rm (3)} $f$ is an isomorphism $\Longleftrightarrow Kerf=\lbrace 1_{G}\rbrace$ and $Imf=H$.
\end{proposition}

\begin{proof}
	{\rm (1)} $\left( \Longrightarrow\right)  \forall a\in Kerf, f(a)=1_{H}=f\left(1_{G}\right)$ since $f$ is a monomorphism, then $a=1_{G}$, that is $Kerf=\left\lbrace 1_{G}\right\rbrace$.
	
	$\left( \Longleftarrow\right) $ If $f(x)=f(y)$, then we have
	$f\left(xy^{-1}\right)=f(x) f(y^{-1} )=f(x) f(y)^{-1}=f(x) f(x)^{-1}=1_{H}$, which implies that $xy^{-1}=1_{G}$. Since $\left(xy^{-1}\right) \alpha(y)=1_{G}\alpha(y)=\alpha^{2}(y)$
	and $\left(xy^{-1}\right) \alpha(x)=\alpha(x)(y^{-1}y)=\alpha(x)(1_G)=\al^{2}(x)$, we have $x=y$. Therefore $f$ is a monomorphism.
	
	{\rm (2)} $Imf$ is a Hom-subgroup of $H$, by the definition of $Imf$, $f$ is an epimorphism if and only if $Imf=H$.
	
	{\rm (3)} It is obvious that $f$ is an isomorphism if and only if $Kerf=\lbrace 1_{G}\rbrace$ and $Imf=H$ by {\rm (1)} and {\rm (2)}.
\end{proof}

\begin{example}
	Let $H$ be a Hom-normal subgroup of a Hom-group $G$. The map $\pi:G\longrightarrow G/H$ is an epimorphism of Hom-groups, called canonical homomorphism, and $Ker\pi=H$.
\end{example}

\begin{proof}
	Since $\pi(xy)=(xy)H=(xH)(yH)=\pi(x)\pi(y)$ and $\pi(\alpha(x))=\alpha(x)H=\tilde{\alpha}(xH)=\tilde{\alpha}(\pi(x))$, we have $\pi$ is a homomorphism of Hom-groups. Obviously, $\pi$ is an epimorphism.
	
	By reference \cite[Lemma 3.2]{Mohammad Hassanzadeh}, for each $x\in H$, we have $\pi(x)=(x)H=H=1_{G/H}$, which implies $x\in Ker\pi$. Thus, $H\subseteq Ker\pi$. For each $x\in Ker\pi$, note that $\pi(x)=(x)H=1_{G/H}=H$, so we have $x\in H$. Thus, $Ker\pi\subseteq H$. Therefore, $Ker\pi=H$.
\end{proof}

\begin{theorem}(Fundamental Theorem of Homomorphisms of Hom-groups)
	Let $G$, $H$ be two Hom-groups. If $f:(G,\beta)\longrightarrow \left( H,\alpha\right) $ is an epimorphism, then $G/Kerf\cong H$.
\end{theorem}

\begin{proof}
	By proposition \ref{eq2}, we have $gKerf=Kerfg$, which implies that $G/Kerf$ is a Hom-group.
	
	Define $\overline{f}:G/Kerf\longrightarrow H$, where $xKerf \longmapsto f(x)$.
	
	First, we verify $\overline{f}$ is well-defined:
	\begin{eqnarray*}
		x Kerf=y Kerf &\Rightarrow & x^{-1} y \in Kerf\\
		&\Leftrightarrow & f(x^{-1}y)=1_{H}\\
		&\Rightarrow & f(x)^{-1} f(y)=1_{H}\\
		&\Rightarrow & \alpha\left(f(x)\right)\left(f(x)^{-1} f(y)\right)=\alpha(f(x)) 1=\alpha^{2} (f(x)).
	\end{eqnarray*}
    Since $\alpha\left(f(x)\right)\left(f(x)^{-1} f(y)\right)=\left(f(x) f(x)^{-1}\right) \alpha(f(y))=1_{H}\alpha(f(y))=\alpha^{2}(f(y))$, we have $\alpha^{2} (f(x))= \alpha^{2}(f(y))$. Therefore $f(x)=f(y)$, which implies that $\overline{f}(\overline{x})=\overline{f}(\overline{y})$.
	
	Moreover, we need to prove $\overline{f}$ is a homomorphism of Hom-groups:
	note that
	$$\overline{f}((xKerf)(yKerf))=\overline{f}(xyKerf)=f(x y)=f(x)f(y)=\overline{f}(xKerf)\cdot \overline{f}(yKerf),$$
	$$\overline{f}(\tilde{\beta}(xKerf))=\overline{f}(\beta(x)Kerf)=f(\beta(x))=\alpha(f(x))=\alpha(\overline{f}(xKerf)),$$ which implies that $\overline{f}\circ \tilde{\beta}=\alpha \circ \overline{f}$. Hence, $\overline{f}$ is a homomorphism of Hom-groups.
	
	Finally, we will prove $\overline{f}$ is a bijective map:
	
	(1) $\overline{f}$ is a surjection: since $f$ is an epimorphism, for any $ h\in H$, there exists $x\in G$ such that $f(x)=h$, then $\overline{f}(x Kerf)=f(x)=h$, hence $\overline{f}$ is a surjection.
	
	(2)$\overline{f}$ is an injection:
	we need to prove $Ker\overline{f}=\lbrace1_{G/Kerf}\rbrace$.

	Suppose that $x Kerf\in Ker \overline{f}$, then $\overline{f}(x Kerf)=1_{H}=f(x)$, that is $f(x)=1_{H}$, so $x \in Ker f$. Therefore, $ x Ker f=Ker f=Ker \overline{f}=\lbrace 1_{G/Ker f}\rbrace$.
\end{proof}

\begin{corollary}
	If  $f:(G,\beta)\longrightarrow \left( H,\alpha\right) $ is a homomorphism of Hom-groups, then $G/Kerf \cong Imf$.
\end{corollary}

Similar to groups, we have the first isomorphism theorem in Hom-groups. Now we provide a lemma before the theorem.

\begin{lemma}\label{410}
	Let $H$ be a Hom-subgroup of $G$ and let $N$ be a Hom-normal subgroup of $G$, then $NH=HN$ is a Hom-subgroup of $G$.
\end{lemma}

\begin{proof}
	For all $n \in N, h \in H$, we have $(hn) \alpha\left(h^{-1}\right) \in N$ by $N\triangleleft G$. There exists $n^{\prime}\in N$ such that
	\begin{gather*}
	 (h n) \cdot \alpha\left(h^{-1}\right)=n^{\prime},\\
		 \left((h n) \alpha\left(h^{-1}\right) \right) \alpha^{2}(h)=n^{\prime}\alpha^{2}(h),\\
		 \alpha(h n)\alpha\left(h^{-1} h\right)=n^{\prime} \alpha^{2}(h),\\
		 \alpha(h n) 1_{H}=n^{\prime}  \alpha^{2}(h),\\
		 \alpha^{2}(h n)=n^{\prime} \alpha^{2}(h),
	\end{gather*}
	Then we have $h n=\alpha^{-2}\left(n^{\prime}\right) h \in N H$, which implies that $NH\subseteq HN$.

	On the other hand, for all $ n \in N, h \in H$, we have $h^{-1}\in H$. Note that $\left(h^{-1} n\right)\alpha(h)=\left(h^{-1} n\right) \alpha( \left(h^{-1}\right)^{-1} ) \in N$ by $ N \triangleleft G$, then there exists $n^{\prime \prime}$ such that
	\begin{gather*}
	     \left(h^{-1} n\right) \alpha(h)=n^{\prime \prime},\\
		 \alpha\left(h^{-1}\right)(n h)=n^{\prime \prime},\\
		 \alpha^{2}(h)\left(\alpha\left(h^{-1}\right)(n h)\right)=\alpha^{2}(h)n^{\prime \prime},\\
		 \alpha\left(h h^{-1}\right) \alpha(n h)=\alpha^{2}(h) n^{\prime\prime},\\
		 \alpha^{2}(n h)=\alpha^{2}(h)  n^{\prime \prime}.
	\end{gather*}
	Then we have $n h=h \alpha^{-2}\left(n^{\prime \prime}\right) \in HN$, which implies that$N H\subseteq H N$. Therefore, $N H=H N$.
	
	In the last, we show that $NH\preceq G$. For any $n h, n^{\prime} h^{\prime} \in N H$, we have
	\begin{eqnarray*}
		(n h)\left(n^{\prime} h^{\prime}\right)^{-1}&=&(n h)\left(h^{\prime -1} n^{\prime -1}\right)\\&=&\alpha(n)\left(h  \alpha^{-1}\left(h^{\prime -1} n^{\prime -1}\right)\right)\\
		&=&\alpha(n)\left(\left(\alpha^{-1}(h) \alpha^{-1}\left(h^{\prime -1}\right)\right) \alpha \alpha^{-1}\left(n^{\prime-1}\right) \right)\\
		&=&\alpha(n)\left(\alpha^{-1}\left(h h^{\prime -1}\right) n^{\prime -1} \right)\in N H.
	\end{eqnarray*}
	
    \noindent So the lemma holds.
\end{proof}

\begin{theorem}(First Isomorphism Theorem)
	Let $H$ and $N$ be Hom-subgroups of a Hom-group $G$ with $N$ normal in $G$. Then
	
	{\rm (1)} $N\cap H$ is a Hom-normal subgroup of $H$.
	
	{\rm (2)} $N$ is a Hom-normal subgroup of $NH$.
	
	{\rm (3)} $H/(N\cap H)\cong NH/N$.
\end{theorem}

\begin{proof}
	{\rm (1)} $\forall h \in H,x \in N\cap H$, we only need to prove $( h x ) \alpha ( h^{-1} ) \in N \cap H$.
	
	Due to $x,h,\alpha^{-1}(h)\in H$, we have $(h x) \alpha\left(h^{-1}\right) \in H$. Note that $N$ is a Hom-normal subgroup of $G$, we have $(h x) \alpha\left(h^{-1}\right) \in N$. Hence, $(h x) \alpha\left(h^{-1}\right) \in N\cap H$.
	
	{\rm (2)} By lemma \ref{410}, $NH$ is a Hom-subgroup of $G$, now we prove that $N$ is a Hom-normal subgroup of $NH$. For all $n h \in N H,x \in N$, we have
	\begin{eqnarray*}
		\left(\left(n h\right) x\right)\alpha\left(\left( n h\right) ^{-1}\right)
		&=&\left(\alpha(n)\alpha(h)\right)\left(x\left(h^{-1} n^{-1}\right)\right)\\
		&=&\alpha^{2}(n)\left(\alpha(h) \alpha^{-1}\left(x\left(h^{-1} n^{-1}\right)\right)\right)\\
		&=&\alpha^{2}(n) \left(\alpha(h)\left(\alpha^{-1}(x)  \alpha^{-1}\left(h^{-1} n^{-1}\right)\right)\right)\\
		&=&\alpha^{2}(n)\left(\left(h \alpha^{-1}(x)\right)\left(h^{-1} n^{-1}\right)\right)\\
		&=&\alpha^{2}(n)\left(\left(\alpha^{-1}\left(h \alpha^{-1}(x)\right) h^{-1}\right) \alpha\left(n^{-1}\right)\right)\\
		&=&\alpha^{2}(n)\left(\left(\left(\alpha^{-1}(h) \alpha^{-2}(x)\right) h^{-1}\right) \alpha\left(n^{-1}\right)\right)\in N.
	\end{eqnarray*}
	
	\noindent Therefore, $N$ is a Hom-normal subgroup of $NH$.
	
	{\rm (3)} Define: $f : H \longrightarrow N H / N\subseteq G / N$, $f\left(h\right)=hN$.
	We need to prove that $f$ is an epimorphism and $Kerf=N \cap H$.
	
	First, $f$ is a homomorphism of Hom-groups: $\forall h_{1}, h_{2}$, we have $f(h_{1}h_{2})=\left(h_{1} h_{2}\right) N=\left(h_{1} N\right)\left(h_{2} N\right)=f\left(h_{1}\right) f\left(h_{2}\right)$ and
	$f(\alpha(h))=\alpha(h) N=\widetilde{\alpha}(h N)=\widetilde{\alpha}(f(h))$. Thus, $f$ is a homomorphism of Hom-groups.
	
	In addition, $f$ is a surjection: we only need to prove $NH/N = \lbrace hN|h\in H\rbrace $. $\forall \left(nh\right)N\in NH/N$, we have
	$$(nh) N=\left(h^{\prime} n^{\prime}\right) N=\left\{\left(h^{\prime} n^{\prime}\right)\alpha(n) | n \in N\right\}=\left\{\alpha\left(h^{\prime}\right)(n^{\prime}n) | n \in N\right\}=\left\{\alpha\left(h^{\prime}\right) n | n \in N\right\}\in\{h N | h \in H\}.$$
	$\forall h N\in HN$, we have $h N=\{h n | n \in N\}=\left\{\left(1_{G}\alpha^{-1}(h)\right)n | n \in N\right\}=\left(1_{G}\alpha^{-1}(h)\right)N \in NH/N$. Hence, $f$ is a surjection.
	
	In the last, we show that $kerf=N \cap H$. Assume that $h\in H$, we have $h\in Kerf \Leftrightarrow f\left(h\right)=1\Leftrightarrow hN=N\Leftrightarrow h\in N\Leftrightarrow h\in N\cap H$, that is $Kerf=N\cap H$.
	
	According to the fundamental theorem of homomorphisms of Hom-groups, $H/(N\cap H)\cong NH/N$.
\end{proof}

\begin{theorem}(Second Isomorphism Theorem)
	Let $\left(M,\al\right),\left(N,\al\right)$ be two Hom-normal subgroups of $\left(G,\al\right)$ and let $\left(N,\al\right)$ be a Hom-subgroup of $\left(M,\al\right)$. Then $\left(M/N,\widetilde\al\right)$ is a Hom-normal subgroup of $\left(G/N,\widetilde\al\right)$ and $\left(G/N\right)/\left(M/N\right)\cong G/M$.
\end{theorem}

\begin{proof}
	$\forall m \in M,g \in G$, we have $\left((g N)(m N)\right) \widetilde{\alpha}\left(g^{-1} N\right)=\left((g N)(m N)\right) \left(\alpha\left(g^{-1}\right) N\right)=\left(\left(g m\right)\alpha\left(g^{-1}\right) \right) N \in M/N$, which implies that $M/N\triangleleft G/N$, so $(G/N)/(M/N)$ is a Hom-group.
	
	Define $f:G/N\longrightarrow G/M$, $f\left(gN\right)=gM$.
	
	(1) It is obvious that $f$ is a surjection.
	
    (2) $f$ is an epimorphism:
	$\forall (g_{1} N)( g_{2} N)$,$$f\left((g_{1} N )( g_{2} N)\right)=f(g_{1} g_{2} N)=\left(g_{1} g_{2}\right) M=(g_{1} M)(g_{2} M)=f(g_{1} N) f(g_{2} N);$$
	$\forall g N\in G/N$,$$f(\widetilde{\alpha}(g N))=f\left(\alpha\left(g\right)N\right)=\alpha\left(g\right)M=\widetilde{\alpha}\left(gM\right)=\widetilde{\alpha}\left(f\left(gN\right)\right).$$
	Hence, $f$ is an epimorphism.
	
	(3) $\forall gN\in G/N, gN\in Kerf$, we have $f\left(gN\right)=gM=M$ if and only if $g\in M$, which implies that $Kerf=M/N$. Since $f$ is a homomorphism of Hom-groups, we have $\left(G/N\right)/\left(M/N\right)\cong G/M$.
\end{proof}

\begin{example}
	Let $\left(G,\alpha\right),\left(G^{\prime},\alpha^{\prime}\right)$ be two Hom-groups. Then
	$\left(G\times G^{\prime},\alpha \times \alpha^{\prime} \right)$ is a Hom-group, where the operation is defined by $\left(g,g^{\prime}\right)\left(h,h^{\prime}\right)=\left(gh,g^{\prime}h^{\prime}\right)$.
\end{example}

\begin{proof}
	Define a map: $\alpha \times \alpha^{\prime}:G\times G^{\prime}\longrightarrow G\times G^{\prime}$, $\alpha \times \alpha^{\prime}\left(g,h\right)=\left(\alpha\left(g\right),\alpha^{\prime}\left(h\right)\right)$, it is obvious that $\alpha \times \alpha^{\prime}$ is a bijective map.
	
	(1) $\alpha \times \alpha^{\prime}$ satisfies the Hom-associativity property: for all $g,h,k\in G, g^{\prime},h^{\prime},k^{\prime}\in G^{\prime}$, we have
	 \begin{align*}
	 	\alpha \times \alpha^{\prime}\left(g, g^{\prime}\right)\left(\left(h, h^{\prime}\right)\left(k, k^{\prime}\right)\right)&=\left(\alpha\left(g\right), \alpha^{\prime}\left(g^{\prime}\right)\right) \left(hk,h^{\prime}k^{\prime}\right)\\&=\left(\alpha \left(g\right)(h k), \alpha^{\prime}\left(g^{\prime}\right)\left(h^{\prime} k^{\prime}\right) \right)\\&=\left(\left(g h\right) \alpha\left(k\right),\left(g^{\prime} h^{\prime}\right)\alpha^{\prime} \left(k^{\prime}\right)\right)\\&=\left(g h, g^{\prime} h^{\prime}\right)\left(\alpha(k), \alpha^{\prime}\left(k^{\prime}\right)\right)\\&=\left(\left(g g^{\prime}\right)\left(h h^{\prime}\right)\right)\left(\alpha(k), \alpha^{\prime}\left(k^{\prime}\right)\right)\\&=\left(\left(g,g'\right)\left(h,h'\right)\right)\alpha \times \alpha^{\prime}\left(k,k'\right).
	 \end{align*}
	
	(2) $\alpha \times \alpha^{\prime}$ is multiplicative: for all $g,h\in G, g^{\prime},h^{\prime}\in G^{\prime}$, we have $\alpha \times \alpha^{\prime}\left(\left(g, h\right) \left(g^{\prime}, h^{\prime}\right)\right)=\alpha \times \alpha^{\prime}\left(\left(g g^{\prime}, h h^{\prime}\right) \right)=\left(\alpha \left(gg^{\prime}\right), \alpha^{\prime}\left(h h^{\prime}\right) \right)=\left(\alpha\left(g\right) \alpha \left(g^{\prime}\right),\alpha^{\prime}\left(h\right) \alpha^{\prime}\left(h^{\prime}\right)\right)=\left(\al\left(g\right),\al'\left(h\right)\right)\left(\al\left(g'\right),\al'\left(h'\right)\right)=\alpha \times \alpha^{\prime}\left(g,h\right)\alpha \times \alpha^{\prime}\left(g',h'\right)$.
	
	(3) Identity element $\left(1_{G},1_{G^{\prime}}\right)$: $\forall\left(g,g^{\prime}\right)\in \left(G,G^{\prime}\right)$,  $$\left(g,g^{\prime}\right)\left(1_{G},1_{G^{\prime}}\right)=\left(g1_{G},g^{\prime}1_{G^{\prime}}\right)=\left(\alpha\left(g\right),\alpha^{\prime}\left(g^{\prime}\right)\right)=\alpha \times \alpha^{\prime}\left(g,g^{\prime}\right)= \left(1_{G},1_{G^{\prime}}\right)\left(g,g^{\prime}\right),$$ and $\alpha \times \alpha^{\prime}\left(1_{G},1_{G^{\prime}}\right)=\left(\alpha\left(1_{G}\right),\alpha^{\prime}\left(1_{G^{\prime}}\right)\right)=\left(1_{G},1_{G^{\prime}}\right)$.
	
	(4) For all $\left(g, g^{\prime}\right)\in \left(G,G^{\prime}\right)$, there exists $\left(g^{-1}, g^{\prime -1}\right)\in \left(G,G^{\prime}\right)$ such that $\left(g, g^{\prime}\right)\left(g^{-1}, g^{\prime -1}\right)=\left(g g^{-1} , g^{\prime} g^{\prime -1}\right)=\left(1_{G},1_{G}\right)=\left( g^{-1}g,g^{\prime -1}g^{\prime}\right)=\left(g^{-1}, g^{\prime -1}\right)\left(g, g^{\prime}\right)$.
	
	In conclusion, $\left(G\times G^{\prime},\alpha \times \alpha^{\prime} \right)$ is a Hom-group.
\end{proof}

\begin{example}
	Let $\left(G,\alpha\right),\left(G^{\prime},\alpha^{\prime}\right)$ be two Hom-semigroups, then
	$\left(G\times G^{\prime},\alpha \times \alpha^{\prime} \right)$ is a Hom-semigroup, where the operation is defined by $\left(g,g^{\prime}\right)\left(h,h^{\prime}\right)=\left(gh,g^{\prime}h^{\prime}\right)$.
\end{example}

\begin{proposition}
	Let $\left(G,\alpha\right),\left(H,\beta \right)$ be two Hom-groups, then
	
	{\rm (1)} There exist two monomorphism: $i_{G}:G\longrightarrow G\times H,g\longmapsto \left(g,1_{H}\right)$ and $i_{H}:H\longrightarrow G\times H,h\longmapsto \left(1_{G},h\right)$ such that $G\cong i_{G}\left(G\right)$ and $H\cong i_{H}\left(H\right)$;
	
	{\rm (2)} $G\times H=i_{G}\left(G\right)i_{H}\left(H\right)$;
	
	{\rm (3)} $i_{G}\left(G\right)\cap i_{H}\left(H\right)=\lbrace\left(1_{G},1_{H}\right)\rbrace$;
	
	{\rm (4)} $\forall x\in i_{G}\left(G\right),\forall y\in i_{H}\left(H\right)$, we have $xy=yx$.
\end{proposition}

\begin{proof}
	{\rm (1)} $\forall a, b \in G$, we have $$i_{G}(a b)=(a b, 1_{H})=(a, 1_{H})(b, 1_{H})=i_{G}(a)i_{G}(b),$$
	$$i_{G}(\alpha(a))=\left(\alpha\left(a\right),1_{H}\right)=\left(\alpha\left(a\right),\beta\left(1_{H}\right)\right)=\alpha \times \beta\left(a,1_{H}\right)=\alpha \times \beta\left(i_{G}\left(a\right)\right).$$
	So $i_{G}$ is a homomorphism of Hom-groups. Similarly, $i_{H}$ is a homomorphism of Hom-groups.
	
In addition, since $a\in Ker\left(i_{G}\right)\Leftrightarrow i_{G}\left(a\right)=1_{G\times H}\Leftrightarrow \left(a,1_{H}\right)=\left(1_{G},1_{H}\right)\Leftrightarrow a=1_{G}\Rightarrow Ker(i_{G})=\left\{1_{G}\right\}$, $i_{G}$ is a monomorphism. Similarly, $i_{H}$ is a monomorphism.
	
Moreover, since $i_{G}\left(1_{G}\right)=1_{G\times H}$, we can get $G=G/Ker(i_{G})\cong i_{G}\left(G\right)=\lbrace i_{G}\left(g\right),g\in G\rbrace$. Similarly, $H\cong i_{H}\left(H\right)$.
	
	{\rm (2)} Since $i_{G}(G)$ and $i_{H}(H)$ are Hom-subgroups of $G\times H$, we have $i_{G}(G)\cdot i_{H}(H)$ is a Hom-subgroup of $G\times H$.
	
	{\rm (3)} For all $(g,h)\in i_{G}(G)\cap i_{H}(H)$, we have  $(g,h)\in i_{G}(G)$ and $(g,h)\in i_{H}(H)$, which implies that $ h=1_{H}$ and $g=1_{G}$. Thus, $(g,h)=(1_{G},1_{H})=1_{G\times H}$.
	
	{\rm (4)} Before that, we need to prove the following conclusion first: for all Hom-groups $G$, $xy=yx$ if and only if $\forall x,y\in G, \alpha^{-1}(xy)x^{-1}=\alpha\left(y\right)$.
	
    Assume that $\alpha^{-1}(xy)x^{-1}=\alpha\left(y\right)$, then \begin{gather*}
    	\left(\alpha^{-1}(xy)x^{-1}\right)\alpha\left(x\right) =\alpha\left(y\right)\alpha\left(x\right),\\
    	 (xy)(x^{-1}x)=\alpha(yx),\\\alpha(xy)=\alpha(yx),\\xy=yx.
    \end{gather*}
	
    Suppose that $xy=yx$, then $\alpha^{-1}(xy)x^{-1}=\alpha^{-1}(yx)x^{-1}=y\alpha^{-1}(xx^{-1})=\alpha(y)$.

    So $xy=yx$ if and only if $\alpha^{-1}(xy)x^{-1}=\alpha\left(y\right)$.
	
	Now we prove {\rm (4)},
	$\forall (g,1_{H})\in i_{G}(G), \forall (1_{G}, h)\in i_{H}(H)$, we have $(\alpha \times \beta)^{-1}(g,1_{H})(1_{G}, h)(g,1_{H})^{-1}=\alpha^{-1}\times \beta^{-1}(\alpha(g),\beta(h))(g^{-1},1_{H})=(g,h) (g^{-1},1_{H})=(\alpha,1_{H})=(1_{G},\beta\left(h\right))=\alpha \times \beta(1_{G},h). $
\end{proof}

\begin{proposition}
	Let $H$, $K$ be Hom-subgroups of $G$. Suppose that $H,K$ satisfy the following three conditions:
	
	{\rm (1)} $G=HK$;
	
	{\rm (2)} $H\cap K=\lbrace1\rbrace$;
	
	{\rm (3)} $\forall h\in H,k\in K$, we have $hk=kh$,
	
\noindent	then $G\cong H\times K$.
\end{proposition}

\begin{proof}
	Define $f:H\times K\longrightarrow G,(h,k)\longmapsto hk$. We claim that $f$ is an isomorphism.
	
	$\forall \left(h,k\right),\left(h^{\prime},k^{\prime}\right)\in H\times K,$ we can get
	\begin{eqnarray*}
		f\left(\left(h,k\right)\left(h^{\prime},k^{\prime}\right)\right)&=& f\left(hh^{\prime},kk^{\prime}\right)\\
		&=& \left(hh^{\prime}\right)\left(kk^{\prime}\right)\\
		&=& \alpha^{-1}\left(\left(h h^{\prime}\right) k\right) \alpha(k)\\
		&=& \left(h\left(\alpha^{-1}\left(h^{\prime} k\right)\right)\right) \alpha\left(k^{\prime}\right)\\
		&=& \left(\alpha^{-1}(h k) h^{\prime}\right) \alpha\left(k^{\prime}\right)\\
		&=&\left(h k\right)\left(h^{\prime} k\right)=f\left(h, k\right) f\left(h^{\prime}, k^{\prime} \right).
	\end{eqnarray*}
	$\forall \left(h,k\right)\in H\times K$, we see that $f(\alpha \times \alpha(h, k))=f(\alpha(h), \alpha(k))=\alpha(h) \alpha(k)=\alpha(h k)=\alpha(f(h, k))$.
	Hence, $f$ is a homomorphism.
	
	It is obvious that $f$ is a surjection by $G=HK$.
	
	$\forall \left(h,k\right)\in kerf$, we have
	\begin{gather*}
		f\left( h,k\right)=hk=1_{G},\\ h=k^{-1}\in H\cap K=\left\lbrace 1_{G}\right\rbrace,\\ h=1_{G}=k,
	\end{gather*}
	then $f$ is an injection. Thus, $G\cong H\times K$.
\end{proof}

\section{Hom-group actions}\label{sec5}
In this section, we introduce a notion of Hom-group action and study some properties of Hom-group actions. To get nice results on Hom-group actions, we need to consider only involutive Hom-groups.

\begin{definition}
Let $(G,\alpha)$ be a Hom-group and $(X,\beta)$ be any non-empty Hom-set. An action of a Hom-group $G$ on $X$ is a map
$$
\theta:~ G\times X\to X,~~~~~~ \theta(g, x)=gx,
$$
satisfying
\begin{enumerate}
\item $1x=\beta(x),~\text{for all} ~x\in X$ and $1$ is the Hom-identity of $G$;
\item $(g_1 g_2)x=\alpha(g_1)(\beta^{-1}(\alpha(g_2)x)),~~~\text{for all}~~~ g_1,g_2\in G,~x\in X.$
\end{enumerate}
\end{definition}
\begin{example}
Let $(G,\alpha)$ be a Hom-group and $(H, \alpha)$ be a Hom-subgroup of $G$. Then there is a regular left $H$ action on $G$ defined by $\theta(h,g)=hg$.
\end{example}

The following is an equivalent definition of Hom-group action.
\begin{proposition}
A Hom-group $(G, \alpha)$ is acting on a Hom-set $(X,\beta)$ if and only if there exist a weak Hom-group homomorphism
$$\psi: G\to \bf{Sym(X)}.$$
\end{proposition}
\begin{proof}
First assume that $G$ is acting on the Hom-set $(X,\beta)$.
For all $g\in G$, we define
\begin{align*}
\phi_g: X\to X,~~x\mapsto \theta(g,x)=gx.
\end{align*}
Note that $$\phi_1=\beta,~~~\text{and}~~~ \phi_{\alpha^{-1}(g g^{-1})}=\phi_g\circ\beta^{-1}\circ\phi_{g^{-1}}=\phi_1.$$
This implies $\phi_g\circ\beta^{-1}\in\bf{Sym(X)}$. Thus, we have $\phi_g\in \bf{Sym(X)}$. Now we define
\begin{align*}
\psi:~ G\to \bf{Sym(X)},~~~g\mapsto \phi_g.
\end{align*}
We show that $\psi$ is a weak Hom-group homomorphism. We have
\begin{align*}
&\psi(1)=\phi_1=\beta,\\
&\psi(g_1g_2)(x)\\
&=\phi_{g_1 g_2}(x)\\
&=(g_1 g_2)x=\alpha(g_1)(\beta^{-1}(\alpha(g_2)x))\\
&=\psi(\alpha(g_1))\circ\beta^{-1}\circ \psi(\alpha(g_2))(x)
\end{align*}
Thus, $\alpha_\beta\circ \psi(g_1 g_2)= \psi(\alpha(g_1)) \psi(\alpha(g_2))$.

Conversely, suppose $\psi: G\to \bf{Sym(X)}$ is a weak Hom-group homomorphism. We show there is an action of $G$ on $X$. We define
\begin{align*}
\theta:~& G\times X\to X,\\
           &(g,x)\mapsto \theta(g,x)=gx := \psi(g)(x),~~~\text{for all}~~~g \in G, x\in X.
\end{align*}
First observe that $1 x= \theta(1,x)=\psi(1)(x)=\beta(x)$. We also have
\begin{align*}
&(g_1 g_2) x\\
&=\theta(g_1 g_2,x)\\
&=\psi(g_1 g_2)(x)\\
&=\psi(\alpha(g_1))\circ\beta^{-1}\circ \psi(\alpha(g_2))(x)\\
&=\alpha(g_1) (\beta^{-1}(\alpha(g_2) x))
\end{align*}
Thus, there is an action of $G$ on $X$.
\end{proof}
\begin{definition}
Let $(X, \beta)$ be a Hom-set equipped with a Hom-group action of $(G, \alpha)$.  Suppose $x$ is an element of $X$. We define an orbit of $x$, denoted by $G^\alpha(x)$ as follows:
\begin{align*}
G^\alpha(x)=\lbrace \beta^{-1}(\alpha(g) x)\mid g\in G\rbrace.
\end{align*}
The stabilizer of $x$ is defined by
\begin{align*}
G^\alpha_x=\lbrace g\in G\mid \alpha(g) x=\beta(x)\rbrace.
\end{align*}
\end{definition}

Note that when Hom-group $(G, \alpha=id)$ acts on the Hom-set $(X, \beta=id)$, we get the usual orbit and stabilizer of $x$ corresponding to a group action.

\begin{lemma}
Suppose $(G, \alpha)$ is an involutive Hom-group acting on a non-empty Hom-set $(X, \beta)$. Then the stabilizer of $x$ is a Hom-subgroup of $G$.
\end{lemma}
\begin{proof}
As $(G, \alpha)$ is an involutive Hom-group, we have $\alpha^2 = \alpha$. First we need to check that $\alpha(G_x^\alpha)\subseteq G_x^\alpha$. Suppose $g\in G_x^\alpha$. Observe that
\begin{align*}
\alpha(\alpha(g)) x=\alpha^2(g) x=\alpha(g) x=\beta(x).
\end{align*}
Thus, $\alpha(g)\in G^\alpha_x$. This implies $\alpha(G_x^\alpha)\subseteq G_x^\alpha$.
Note that $1\in  G^\alpha_x$ as $1 x=\beta(x)$.
Let $g_1, g_2 \in G^\alpha_x$. We have
\begin{align*}
\alpha(g_1 g_2)x & = (\alpha(g_1) \alpha(g_2)) x \\
                                  & = \alpha (\alpha(g_1))(\beta^{-1}\alpha(\alpha(g_2))x) \\
                                  & = \alpha^2(g_1)(\beta^{-1}\alpha^2(g_2)x) \\
                                  & =\alpha(g_1)(\beta^{-1}\alpha(g_2)x) \\
                                  & =\alpha(g_1)((\beta^{-1}\beta) x) \\
                                  & =\alpha(g_1) x \\
                                  & =\beta(x).
                                  \end{align*}
Therefore, $g_1  g_2 \in  G^\alpha_x$. Suppose $g\in  G^\alpha_x$, we have
\begin{align*}
\beta(x)=1 x=(g^{-1} g) x=\alpha(g^{-1}) (\beta^{-1}(\alpha(g) x))=\alpha(g^{-1}) x.
\end{align*}
Thus, $g^{-1}\in G_x^\alpha$. Therefore, $G^\alpha_x$ is a Hom-subgroup of $G$.
\end{proof}
\begin{theorem}{(Orbit-Stabilizer theorem)}
Let $(G,\alpha)$ be a finite involutive Hom-group. Suppose the Hom-group $G$ is acting on non-empty Hom-set $(X, \beta)$. Then for all $x\in X$,
\begin{align*}
\mid G\mid=\mid G^\alpha_x\mid~ \mid G^\alpha(x)\mid.
\end{align*}
\end{theorem}
\begin{proof}
We define a map
\begin{align*}
\phi: ~&\frac{G}{G^\alpha_x}\to G^\alpha(x), ~~~\phi(gG_x^\alpha) :=\beta^{-1}(\alpha(g)x)~~~.
\end{align*}
We show that $\phi$ is a bijection.

Suppose $\phi(g_1G_x^\alpha)=\phi(g_2G_x^\alpha)$. We have
\begin{align*}
&\beta^{-1}(\alpha(g_1)x)=\beta^{-1}(\alpha(g_2)x),\\
&\alpha(g^{-1}_1)(\beta^{-1}(\alpha(g_1)x))=\alpha(g^{-1}_1)(\beta^{-1}(\alpha(g_2)x)),\\
&(g^{-1}_1 g_1)x=(g^{-1}_1 g_2)x,\\
&ex=(g^{-1}_1 g_2)x,\\
&\beta(x)=\alpha(\alpha^{-1}(g^{-1}_1 g_2))x.
\end{align*}
This implies $\alpha^{-1}(g^{-1}_1 g_2)\in G_x^\alpha$. As $G_x^\alpha$ is a Hom-subgroup, $g^{-1}_1 g_2\in G_x^\alpha$. This implies
$g_1G_x^\alpha = g_2G_x^\alpha.$
Thus, $\phi$ is one-one. Clearly, $\phi$ is onto as $\phi(gG_x^\alpha)=\beta^{-1}(\alpha(g)x)$. Therefore, $\phi$ is a bijection.
\end{proof}

Suppose  Hom-group $(G,\alpha)$ is acting on a Hom-set $(X,\beta)$. We define a relation $\sim$ on $X$ as $x\sim y$ if $x=\beta^{-1}(\alpha(g)y)$ for some $g\in G$.

Note that relation $\sim$ is reflexive as $x=\beta^{-1}(\alpha(1)x)$. Suppose $x\sim y$, this implies, $x=\beta^{-1}(\alpha(g) y)$ for some $g\in G$. Observe that
\begin{align*}
&\alpha(g^{-1}) x=\alpha(g^{-1}) (\beta^{-1}(\alpha(g) y)),\\
&\alpha(g^{-1}) x=(g^{-1} g) y=1 y=\beta(y),\\
&y=\beta^{-1}(\alpha(g^{-1}) x).
\end{align*}
This implies $y\sim x$. Thus, the relation $\sim$ is symmetric.

Suppose $x\sim y$ and $y\sim z$. We show $x\sim z$. For some $g_1,g_2\in G$, we have
\begin{align*}
x=\beta^{-1}(\alpha(g_1) y),~~~ y=\beta^{-1}(\alpha(g_2) z).
\end{align*}
Observe that,
\begin{align*}
x=\beta^{-1}(\alpha(g_1) (\beta^{-1}(\alpha(g_2) z))=\beta^{-1}((g_1 g_2) z)=\beta^{-1}(\alpha(\alpha^{-1}(g_1g_2)) z).
\end{align*}
This implies the relation is transitive.
Thus, the relation $\sim$ is an equivalence relation. Hence, the relation $\sim$ partition $X$ into its equivalence classes. Note that an equivalence class is nothing but an orbit $G^\alpha(x)$ for some $x\in X$. Being an equivalence class, orbits partitions Hom-set $X$ into pairwise disjoint subsets.

\section{First Sylow  theorem for Hom-groups}\label{sec6}
In this final section, we prove first Sylow theorem for Hom groups using Hom-group actions.
\begin{lemma}\label{counting}\cite{isaacs}
Let $p$ be a prime number and $k\geq 0$ and $m\geq 1$ be integers. Then
\begin{align*}
{p^km\choose p^k}\equiv m~~~\text{mod p}.
\end{align*}
\end{lemma}
\begin{theorem}
\label{first sylow}
Let $(G,\alpha)$ be an involutive finite Hom-group, and let $p$ be a prime divisor of the order of $G$. Then $G$ has a $p$-sylow Hom-subgroup.
\end{theorem}
\begin{proof}
Let $\mid G\mid=p^km$, where $k\geq 0$ and $p$ does not divide $m$. Let $S$ be the collection of all subsets of $G$ having number of elements $p^k$.
Define a bijective map on $S$ as follows:
$$\beta: S\to S,~~~\beta(A)=\alpha(A) = \lbrace \alpha(a) \mid a\in A \rbrace.$$
Thus, $(S,\beta)$ is a Hom-set.
Now we first show $G$ acts on $S$ by regular left multiplication.
\begin{align*}
\theta:&~G\times S\to S,\\
&(g, A)\mapsto g A=\lbrace g a\mid a\in A\rbrace.
\end{align*}
We have
\begin{enumerate}
\item $1 A=\lbrace 1 a\mid a\in A\rbrace=\lbrace \alpha(a)\mid a\in A\rbrace=\alpha(A)=\beta(A)$,
\item $(g_1 g_2) A=\lbrace (g_1 g_2) a\mid a\in A\rbrace.$
\end{enumerate}
On the other hand, we have
\begin{align*}
& \alpha(g_2)A=\lbrace \alpha(g_2) a\mid a\in A\rbrace,\\
 &\beta^{-1}(\alpha(g_2) A)=\lbrace \alpha^{-1}(\alpha(g_2) a)\mid a\in A\rbrace=\lbrace g_2\alpha^{-1}(a)\mid a\in A\rbrace,\\
 &\alpha(g_1)\beta^{-1}(\alpha(g_2) A)=\lbrace \alpha(g_1) (g_2\alpha^{-1}(a))\mid a\in A\rbrace=\lbrace (g_1 g_2) a\mid a\in A\rbrace.
\end{align*}
Therefore, $G$ acts on $S$ by regular left multiplication. This action partition $S$ into disjoint orbits and $|S|$ is the sum of the orbit sizes. Using Lemma \ref{counting},
\begin{align*}
|S|={p^km\choose p^k}\equiv m\neq 0 ~~~\text{(mod p)}.
\end{align*}
This means $|S|$ is not divisible by $p$. Hence, there is some orbit $G^\alpha(X)$ such that $|G^\alpha(X)|$ is not divisible by $p$.

Suppose $Y\in G^\alpha(X)$ and $H=G_Y^\alpha$ be the stabilizer of $Y$ in $G$. From the orbit-stabilizer theorem,
\begin{align*}
|G^\alpha(X)|=\frac{|G|}{|G_Y^\alpha|}.
\end{align*}
Since $p^k~\mid |G|$ and $p~\nmid |G^\alpha(X)| $, this implies, $p^k~\vert~ |G_Y^\alpha|$. Thus, $p^k$ divides $|H|$. This implies $p^k\leq |H|$. Since $H$ is a stabilizer of $Y$ under left multiplication, if $y\in Y$, then $Hy\subseteq Y$. Now,
\begin{align*}
|H|=|Hy|\leq |Y|=p^k.
\end{align*}
Therefore, $|H|=p^k$. As stabilizer $H$ is a Hom-subgroup of $G$  of highest power of $p$, $H$ is a $p$-Sylow Hom-subgroup of $G$.
\end{proof}


\end{document}